\documentclass[reqno]{amsart}
\usepackage[T1]{fontenc}
\usepackage{dsfont}
\usepackage{mathrsfs}
\usepackage[colorlinks, citecolor=blue, linkcolor=blue]{hyperref}%³¬Á´½Ó
\usepackage{xcolor}
\usepackage[a4paper,asymmetric]{geometry}
\usepackage{mathscinet}
\usepackage{latexsym}
\usepackage{amsthm}
\usepackage{amssymb}
\usepackage{amsfonts}
\usepackage{amsmath}
\usepackage{longtable}
\usepackage{graphicx}
\usepackage{multirow}
\usepackage{multicol}
\usepackage{amsfonts, amsmath}
\usepackage{latexsym,bm,amsfonts,amssymb,pifont,mathbbol,bbm}
\usepackage{verbatim}

\newtheorem{theorem}{Theorem}[section]
\newtheorem{thm}[theorem]{Theorem}

\newtheorem{lem}[theorem]{Lemma}
\newtheorem{remark}[theorem]{Remark}
\newtheorem{proposition}[theorem]{Proposition}

\newtheorem{corollary}[theorem]{Corollary}
\newtheorem{hyp}[theorem]{HYPOTHESIS}
\theoremstyle{definition}

\newtheorem{defn}[theorem]{Definition}

\theoremstyle{remark}
\numberwithin{equation}{section}

 \DeclareMathAlphabet{\mathpzc}{OT1}{pzc}{m}{it}
 \DeclareMathAlphabet{\mathsfsl}{OT1}{cmss}{m}{sl}

  \newcommand{\FH}{\mathfrak{H}}

\newcommand{\dif}{\mathrm{d}}

\newcommand{\abs}[1]{\left\vert#1\right\vert}

 \newcommand{\innp}[1]{\langle {#1}\rangle}
\newcommand{\Be}{\begin{equation}}
\newcommand{\Ee}{\end{equation}}
\newcommand{\Bs}{\begin{split}}
\newcommand{\Es}{\end{split}}
\newcommand{\Bes}{\begin{equation*}}
\newcommand{\Ees}{\end{equation*}}
\newcommand{\BT}{\begin{thm}}
\newcommand{\ET}{\end{thm}}
\newcommand{\Bp}{\begin{proof}}
\newcommand{\Ep}{\end{proof}}
\newcommand{\BL}{\begin{lem}}
\newcommand{\EL}{\end{lem}}
\newcommand{\BP}{\begin{proposition}}
\newcommand{\EP}{\end{proposition}}
\newcommand{\BC}{\begin{corollary}}
\newcommand{\EC}{\end{corollary}}
\newcommand{\BR}{\begin{remark}}
\newcommand{\ER}{\end{remark}}
\newcommand{\BD}{\begin{defn}}
\newcommand{\ED}{\end{defn}}
\newcommand{\BI}{\begin{itemize}}
\newcommand{\EI}{\end{itemize}}

\allowdisplaybreaks

\begin{document}
\title{Parameter estimations for the Gaussian process with drift at discrete observation}
\author[]{Luo Shifei}
 \address{College of Mathematics and Information Science, Jiangxi Normal University, Nanchang, 330022, Jiangxi, China}
\email{}
 %\author[]{}
 %\address{ School of Mathematical and Statistical Sciences, Arizona State University,\,Arizona, USA. }%(Corresponding author.)}
 %\email {}
%-------------------------------------------------------------------------------------------------------
\begin{abstract}

This paper first strictly proves that the growth of the second moment of a large class of Gaussian processes is not greater than $ct^{\beta}$($\beta<2$) and the covariance matrix is strictly positive definite. Under these two conditions, the maximum likelihood estimators of the mean and variance of such classes of drift Gaussian process are strongly consistent under broader growth of $t_n$. At the same time, the bivariate asymptotic normal distribution and the Berry-Ess\'{e}en bound of estimators are obtained by using the Stein$^{,}$s method via Malliavian calculus.\par

{\bf Keywords:} Test the hypothesis; Ornstein-Uhlenbeck process; Gaussian process; estimator.\\

\end{abstract}
\maketitle
\pagestyle{empty}
\section{Introduction}\label{sec1}
It is worth considering the statistical inference for stochastic differential equation (SDE) driven by the  general Gaussian process below
\begin{equation}\label{M1}
\mathrm{d}X_t=\mu\mathrm{d}t+\sigma\mathrm{d}G_t,\quad X_0=0,\quad t\geq0,
\end{equation}
where $\{G_t,t\geq0\}$ is a one-dimensional zero-mean Gaussian process. The covariance $R(s,t)=\mathbb{E}(G_sG_t)$ is a real-valued function on $[0,\infty ) \times [0,\infty)$ that is symmetric and positive definite.
%the covariance function $R(s,t)=\mathbb{E}(G_t G_t)$ of which is a symmetric real-valued function as defined on $[0,\infty ) \times [0,\infty)$.
 $\mu$ and $\sigma$ are unknown parameters, while $X_t$ can be treated as a time series with a definite trend.\par
Model (\ref{M1}) was extensively applied in various fields, including the finance, hydrology, and network data analysis. Norros.\cite{B1} first made reference model (\ref{M1}) as a traffic model and applied it to construct the stationary storage model. It was explored how its quality of service parameters relates to system parameters and the lower bound of tail distribution of storage model, where the input data is described by fractional Brownian motion with the Hurst parameter $H <\frac{1}{2}$. Duncan et al.\cite{D1} applied model \eqref{M1} to model the queue length of a single channel FIFO fluid queue with long-range  dependent input and constant service rate, so as to explore
the asymptotic behavior as well as the upper and lower bounds of the tail distribution for queuing models with fractional Brownian motion inputs. In these two papers, $\mu>0$ is taken as the average input rate and $\sigma$ is treated as the fluctuation coefficient. In paper \cite{E6}, model (\ref{M1}) was applied to construct the reflection fractional Brownian motion with drift coefficient, with the strong consistency and asymptotic normality of the estimator $\hat{\mu}_n$ of the parameter $\mu$ obtained, where $\mu$ denotes the average input rate minus the service rate.\par
In general, the parameters $\mu$ and $\sigma$ are unknown and can be estimated according to the observed value, which leads to the problem of parameter estimation about $\mu$ and $\sigma$ as extensively studied in recent years. Xiao et al.\cite{E3} applied model \eqref{M1} to the Chinese stock market when $G_t$ is the fractional Brownian motion of $H<1/2$, and showing the strong consistency of the statistics of $\mu$ and $\sigma$. We note that the power variation method can be applied to estimate the parameter $\sigma>0$ (see \cite{E7},\cite{E8}). Stibrek.\cite{E19} studied the hypothesis testing problem of $\mu$ when $\sigma=1$. In paper \cite{E5}, the strong consistency and asymptotic normality of the least squares estimation of $\mu$ were studied in drift fractional Brownian motion model when the parameter $\sigma=1$ and $H>1/2$
.\par
It is assumed that the process $X$ is observed at discrete time $\mathbf{t}=(t_1,t_2,\cdots,t_n)^{\prime}$, where $t_1<t_2<\cdots<t_n$, then the observed $\mathbf{X}=(X_{t_1},X_{t_2},\cdots,X_{t_n})^{\prime}$ is a random vector.
We introduce the notation
\begin{equation}
\mathbf{X}=\mu \mathbf{t}+\sigma \mathbf{G_t},
\end{equation}
where $\mathbf{G_t}=(G_{t_1},G_{t_2},...,G_{t_n})^{\prime}$ and $\mathbf{G_t}\sim N_n(\mathbf{0},\mathbf{V})$.\par
If the covariance matrix $\mathbf{V}$ of random vector $\mathbf{X}$ is strictly positive definite, there is a joint density function:
\begin{equation}
f(\mathbf{X})=\frac{1}{(\sqrt{2 \pi \sigma^{2}})^{n}\lvert \mathbf{V}\rvert^{1/2}}\exp (-\frac{1}{2 \sigma^{2}}(\mathbf{X}-\mu \mathbf{t})^{\prime}\mathbf{V}^{-1}(\mathbf{X}-\mu \mathbf{t}))
\end{equation}
And the exact maximum likelihood estimates of $\mu$ and $\sigma^2$ can be obtained in \cite{E9}.
\begin{equation}
\hat{\mu}_n=\frac{\mathbf{t}^{\prime} \mathbf{V}^{-1} \mathbf{X}}{\mathbf{t}^{\prime} \mathbf{V}^{-1} \mathbf{t}},
\end{equation}
\begin{equation}
\hat{\sigma}^{2}_n=\frac{1}{n} \frac{\left(\mathbf{X}^{\prime} \mathbf{V}^{-1} \mathbf{X}\right)\left(\mathbf{t}^{\prime} \mathbf{V}^{-1} \mathbf{t}\right)-\left(\mathbf{t}^{\prime} \mathbf{V}^{-1} \mathbf{X}\right)^{2}}{\mathbf{t}^{\prime} \mathbf{V}^{-1} \mathbf{t}}
\end{equation}
\begin{remark}
According to the factorization theorem, we know that $(\hat{\mu}_n,\hat{\sigma}^{2}_n)$ is a sufficient statistic of $(\mu,\sigma^2)$.
\end{remark}
When $G_t$ represents fractional Brownian motion with Hurst parameters $H\in(0,1)$,  Hu et al.\cite{E9} studied almost everywhere convergence, central limit Theorem and Berry-Ess\'{e}en bound of estimators $\hat{\mu}_n$ and $\hat{\sigma}^{2}_n$. the case of drift sub-fractional Brownian motion with Hurst parameter $H>1/2$ was study in \cite{E10}.
Bertin et al.\cite{d9} constructed the incomplete maximum likelihood estimation for the drift parameter $\mu$ as follows
%In paper \cite{d9}, the incomplete maximum likelihood estimation of $\mu$ and the asymptotic properties of $\hat{\mu}_n$ are presented
$$\hat{\mu}=N\frac{\sum_{i,j=1}^{N^{\alpha}} j \Sigma_{i, j}^{-1} Y_{i}}{\sum_{i, j=1}^{N^{\alpha}} i j \Sigma_{i, j}^{-1}}.$$%\qquad\hat{\sigma}_{B}=\sqrt{\frac{\sum_{i=0}^{N^{\alpha}-1}\left(Y_{(i+1) h}-Y_{i h}\right)^{2}}{N^{\alpha} h^{2 H}}}.$$
and considered asymptotic properties of $\hat{\mu}_n$

In this paper, the asymptotic properties of the above-mentioned estimator $\hat{\mu}_n$ and $\hat{\sigma}_n^2$ are studied under the condition that the general Gaussian process satisfies two hypothesis as follows.

\begin{hyp}\label{as01}
 For any $t\geq0$, there exists a constant $c$, such that $\mathbb{E}(G_t^2)\leq ct^{\beta}$.
\end{hyp}
\begin{hyp}\label{as02}
 The covariance matrix $\mathbf{V}=(R(t_i,t_j))_{n\times n}$ is a strictly positive definite matrix.
\end{hyp}
Theorem \ref{D1} presents fractional Brownian motion, sub-fractional Brownian motion, and some other zero-mean Gaussian processes, and all of which satisfy the aforementioned  hypothesis.  Theorem \ref{D4} shows the statistical properties of this class of zero-mean Gaussian processes .
\begin{thm}\label{D1}
 The Gaussian process with the following covariance functions satisfies hypothesis \ref{as01} and \ref{as02}:
\begin{enumerate}
\item $R(t,s)=\begin{cases}\frac{\Gamma(1-K)}{K(2-K)}\left[t^{2HK}+s^{2HK}-(t^{2H}+s^{2H})^{K}\right]&K \in(0,1)\\ \frac{\Gamma(2-K)}{K(K-1)}\left[(t^{2H}+s^{2H})^{K}-t^{2HK}-s^{2HK}\right]&K\in(1,2)\end{cases},$\\
$H\in(0,1), H K\in(0,1)$.
\item $R(s,t)= \frac{1}{2}(s^{2H} + t^{2H} - (t-s)^{2H}),  H\in(0,1).$
\item $R(s,t)=\left(t+s\right)^{2 H }-(t-s)^{2H}, H\in(0,1).$
\item $R(t,s)=s^{2H}+t^{2H}-\frac{1}{2}\left((s+t)^{2H}+(t-s)^{2H}\right),\, H\in(0,1).$
\item $R(s,t)=(s+t)^{2H}-(\max(s,t))^{2H},H\in(0,\frac{1}{2}). $
\item $R(s, t)=(\max (s, t))^{2H}-(t-s)^{2H},H\in(0,\frac{1}{2}).$
\item $R(t,s)=\frac{1}{2}\left[t^{2H}+s^{2H}-K(t+s)^{2H}-(1-K)\lvert t-s\rvert^{2H}\right], H\in(0,\frac{1}{2}),K\in(0,1).$
\item $R(t,s)=\frac{1}{2^K}\left((s^{2H}+t^{2H})^K - \lvert t-s\rvert^{2HK}\right), H\in (0, 1),K \in (0, 1),HK\in(0,1)\,\,or\,\, H\in (0, 1),K \in (1, 2),HK\in(0,1).$
\item$ R(t,\, s)= (s^{2H}+t^{2H})^{K}-\frac12 \big[(t+s)^{2HK} + \lvert t-s\rvert^{2HK} ], H\in(0,1),K\in(0,1),HK\in(0,1) \,\,or\,\, H\in(0,\frac{1}{2}),K\in(0,1).$
\end{enumerate}
\end{thm}

\begin{remark}
Despite many zero-mean Gaussian processes that satisfy hypothesis \ref{as01} and \ref{as02}, not all zero-mean Gaussian processes satisfy both of them simultaneously. For example, if the zero-mean Gaussian process has  covariance function $R(s,t) = \abs{st}^{2H}(H > 0)$ or $R(s,t) = cos(s-t)$, it does not satisfy both hypothesis \ref{as01} and \ref{as02}.
\end{remark}
\begin{thm}\label{D4}
\begin{enumerate}
\item When hypothesis \ref{as01} and \ref{as02} are satisfied and assume $c_{1}n^{\alpha}\leq t_n \leq c_{2}n^{\alpha}$, then, the estimators $\hat{\mu}_n$ and $\hat{\sigma}^2_n$ are strongly consistent, that is,
\begin{equation}
\hat{\mu}_n\stackrel{a.s.}{\longrightarrow}\mu \qquad as\qquad  n\rightarrow \infty,
\end{equation}
\begin{equation}
\hat{\sigma}^2_n\stackrel{a.s.}{\longrightarrow}\sigma^2  \qquad as\qquad  n\rightarrow \infty,
\end{equation}
where, $\alpha>0$, $c_1$ and $c_2$ are fixed constants.
\item  If the Gaussian process satisfies hypothesis \ref{as02}, then
$$
\left[\begin{array}{l}
Y_{n} \\
Q_{n}
\end{array}\right] \stackrel{l a w}{\longrightarrow} N_{2}\left(\mathbf{0}, \mathbf{I}\right)\quad as \quad n\longrightarrow\infty
$$
where~$Y_{n}=\frac{1}{\sigma^2}\sqrt{\mathbf{t}^{'}\mathbf{V}^{-1}\mathbf{t}}(\hat{\mu}_{n}-\mu)%\stackrel{l a w}{\longrightarrow} N(0,\sigma^2)\quad n\rightarrow\infty
,\,Q_{n}=\frac{1}{\sigma^{2}} \sqrt{\frac{n}{2}}\left(\hat{\sigma}_{n}^{2}-\sigma^{2}\right),$\,  with $\mathbf{I}$ denoting the second order identity matrix.
\item If the Gaussian process satisfies hypothesis \ref{as02}, the following properties are considered true:\\
\begin{enumerate}
\item $\sup\limits _{z \in \mathbf{R}}\lvert P\left(\bar{Q}_{n} \leq z\right)-\Phi(z)\rvert \leq \frac{\sqrt{2 n-1}}{n}$, where ~$\Phi(z)=\int_{-\infty}^{z} \frac{1}{\sqrt{2 \pi}} e^{-x^{2} / 2}\mathrm{d} x$.\\
\item $\frac{n}{\sqrt{2 n-1}}\left(P\left(\bar{Q}_{n} \leq z\right)-\Phi(z)\right) \rightarrow\infty,-\frac{\Phi^{(3)}(z)}{3}$, for every $z\in R$, as $n\rightarrow\infty $, where $\Phi^{(3)}(z)=\left(z^{2}-1\right) \frac{1}{\sqrt{2 \pi}} e^{-x^{2} / 2}$ is the third-order derivative of $\Phi(z)$.\\
\item There exists a constand $C\in(0,1)$ and an integer $n_0\geq1$, such that
$$C<\frac{n}{\sqrt{2 n-1}} \sup _{z \in \mathbf{R}}\lvert P\left(\bar{Q}_{n} \leq z\right)-\Phi(z)\rvert\leq 1.$$\\
for every $n\geq n_0$, where $\bar{Q}_{n} =Q_{n}-\mathbb{E}\left(Q_{n}\right)$.
\end{enumerate}
\end{enumerate}
\end{thm}
\begin{remark}
(a).By combining Theorem \ref{D1} and Theorem  Theorem \ref{D4}, it can be known that in Theorem \ref{D1}, the estimators $\hat{\mu}_n$ and $\hat{\sigma}_n^2$  of the 9 types of Gaussian processes with drift terms  have properties  of Theorem \ref{D4}, including the cases of fractional Brownian motion and sub-fractional Brownian motion as proved in the past (b).Compared to \cite{E9} and \cite{E10}, the first property in Theorem \ref{D4} leads to a broader growth of $t_n$ and the second property takes into account the asymptotic property of random vector $(\hat{\mu}_n,\hat{\sigma}_n^2)^{\prime}$.
\end{remark}
The rest of this  paper is organized as follows. In section 2,  the tools used to prove these results are introduced. In section 3, a detailed proof of the main results is presented. In the future, model (\ref{M1}) will be extend  to the case with intercept term and nonlinearity.
\section{Preliminary}\label{I3}
\subsection{ Zero-Mean Gaussian Processes}
\begin{enumerate}
\item Gaussian process $$X^{K}(t)=\int_{0}^{\infty}\left(1-e^{-\theta t}\right) \theta^{-\frac{1+K}{2}} \mathrm{d} W_{\theta},$$
where $W_{\theta}$ is Brownian motion, parameters $H\in(0,1),K\in(0,2)$. Initially, this process is proposed by Lei and Nualart in \cite{A10}. The value range of $K$ is extended from $(0,1)$ to $(0, 2)$ in Bardina and Bascompte \cite{A11}. When $X^{H,K}(t)=X^{K}(t^{2H})$, the covariance function of the process is presented by 1 of Theorem \ref{D1}.
\item The covariance function of fractional Brownian motion $\{B^{H}(t),t\geq0\}$ is presented by 2 of Theorem \ref{D1}.
\item Gaussian process
$$\left(\mathbb{U}(t)\right) \stackrel{f d d}{=} \left(\mathbb{B}^{\alpha / 2}(t)-\mathbb{B}^{\alpha / 2}(-t)\right),$$
the corresponding covariance function is presented by 2 of Theorem \ref{D1}, where $\mathbb{B}^{\alpha / 2}$ represents the two-sided fractional Brownian motion with parameter $\alpha/2=H\in(0,1)$  (see \cite{A17}).
\item The covariance function of sub-fractional Brownian motion $\{S^{H}(t),t\geq0\} $ is presented by 4 of Theorem \ref{D1}.
\item The covariance function of the Gaussian process $\{N(t),t\geq0\}$ is presented by 5 of Theorem \ref{D1}(see \cite{A13},\cite{A17}).
\item The covariance function of the Gaussian process $Z(t)=e^{-HKt}B^{H,K}(e^t)$ is presented by 6 of Theorem \ref{D1}(see \cite{A13}).
\item Gaussian process
$$M^{H, K}(t)=\sqrt{1-K}B^{H}(t)+\sqrt{\frac{HK}{\Gamma(1-2H)}}X^{2H}(t),$$
the corresponding covariance function is presented by 7 of Theorem \ref{D1} ,as obtained by combining fractional Brownian motion and $X^K(t)$ process (see \cite{A14}), where the parameters $H\in(0,\frac{1}{2})$,$K\in(0,1)$.
\item The covariance function of bi-fractional Brownian motion $\{B^{H,K}(t),t\geq0\}$ is presented by 8 of Theorem \ref{D1}, where the parameters $H\in (0, 1),\,K \in (0, 1),\,H K\in(0,1)$ (see \cite{A16}) or $H\in (0, 1),\,K \in (1, 2),\,H K\in(0,1)$ (see \cite{A19}).
\item The covariance function of the Gaussian process $\{S^{H,K}(t),t\geq0\}$ is presented by 9 of Theorem \ref{D1}, which is referred to as sub-bifractional Brownian motion with parameters $K\in(0,1)$  (see \cite{A20}), the parameter $K\in(1,2)$ is referred to as generalized sub-fractional Brownian motion (see \cite{A21}).
\end{enumerate}
\subsection{Gaussian Space}
Given a complete probability space $(\Omega, \mathscr{F}, P)$, the filtration $\mathscr{F}$ is generated by the Gaussian family $G$, where $G=\{G_t, t\in[0, T]\}$ is zero-mean Gaussian process with  covariance function
\begin{equation}
\mathbb{E}(G_{t}G_{s})=R(s,t),s,t\in[0,T].
\end{equation}
Suppose in addition that the covariance function $R$ is continuous. Let $\mathscr{E}$ denote the space of all real valued step functions on $[0,T]$. The Hilbert space $\mathfrak{H}$ is defined as the closure of $\mathscr{E}$ endowed
with the inner product
\begin{align*}
\innp{\mathbf{1}_{[a,b)},\,\mathbf{1}_{[c,d)}}_{\FH}=\mathbb{E}\big(( G_b-G_a)( G_d-G_c) \big).
\end{align*}

We denote $G=\{G(h), h \in \mathfrak{H}\}$ as the isonormal Gaussian process on the probability space $(\Omega, \mathcal{F}, P)$, indexed by the elements in the Hilbert space $\mathfrak{H}$. In other words, $G$ is a Gaussian family of random variables such that
   $$\mathbb{E}(G) = \mathbb{E}(G(h)) = 0, \quad \mathbb{E}(G(g)G(h)) = \langle g,h \rangle_{\mathfrak{H}} \quad g, h \in \mathfrak{H}.$$

The following proposition comes from Theorem 2.3 of \cite{E12}, which gives a concrete representation of the inner product operation in Hilbert space $\mathfrak{H}$.
\begin{proposition}\label{main prelim}
Denote $\mathcal{V}_{[0,T]}$ as the set of bounded variation functions on $[0,T]$. Then $\mathcal{V}_{[0,T]}$ is dense in $\FH$  and we have
\begin{align} \label{innp fg30}
\innp{{f,g}}_{\FH}=\int_{[0,T]^2} R(t,s) \nu_f( \dif t) \nu_{g}( \dif s),\qquad \forall f,\, g\in \mathcal{V}_{[0,T]},
\end{align}
where $\nu_{g}$ is the Lebesgue-Stieljes signed measure associated with $g^0$ defined as
\begin{equation}
g^0(x)=\left\{
      \begin{array}{ll}
 g(x), & \quad \text{if } x\in [0,T] ;\\
0, &\quad \text{otherwise }.
 \end{array}
\right.
\end{equation}\par
\end{proposition}

We will introduce some elements of the Malliavin calculus associated with $G$ (please refer to \cite{A18} for details). Let $C_b^{\infty}(R^{n})$ be the class of infinitely differentiable functions $f:R^n\rightarrow R$ such that $f$ and all its partial derivatives have polynomial growth order. We denote by $S$ the class of smooth and cylindrical random variables of the form
$$F=f( G(\varphi_1),\cdots,G(\varphi_n)),$$
where $n\geq1$, $\varphi_{i}\in \mathfrak{H},\, i=1,2,\cdots,n $ and $f\in C_b^{\infty}(R^n)$. \par
The Malliavin derivative operator D of a smooth and cylindrical random variable $F=f( G(\varphi_1),\cdots,G(\varphi_n))$ of the form is defined as the $\mathfrak{H}$-valued random variable
$$DF=\sum_{i=1}^{N} \frac{\partial f}{\partial x_{i}}\left(G\left(\varphi_{1}\right), \cdots, G\left(\varphi_{n}\right)\right) \varphi_{i}(t)\in L^2(\Omega,\mathfrak{H}),$$
In particular, $D_{s} G_{t}=I_{[0, t]}(s)$.\par
 $\mathbb{D}^{p}_{l,p}$ denotes the closure of the set of smooth random variables with respect to the norm
$$\|F\|_{l,p}^{p}=\mathbb{E}\left(F^{p}\right)+\mathbb{E}\sum^l_{k=1}\left[\|D F\|_{\mathfrak{H}\bigotimes k}^{2}\right].$$

Denote $\mathfrak{H}^{\otimes p}$ and $\mathfrak{H}^{\odot p}$ as the $p$th tensor product and the $p$th symmetric tensor product of the Hilbert space $\mathfrak{H}$. Let $\mathcal{H}_p$ be the $p$th Wiener chaos with respect to $G$. It is defined as the closed linear subspace of $L^2(\Omega)$ generated by the random variables $\{H_p(G(h)): h \in \mathfrak{H}, \ \|h\|_{\mathfrak{H}} = 1\}$, where $H_p$ is the $p$th Hermite polynomial defined by
\begin{equation}
H_p(x)=\frac{(-1)^p}{p!} e^{\frac{x^2}{2}} \frac{d^p}{dx^p} e^{-\frac{x^2}{2}}, \quad p \geq 1,
\end{equation}
and $H_{o}(x)=1$.\par
We have the identity $I_p(h^{\otimes p})=H_p(G(h))$ for any $h \in \mathfrak{H}$, where $I_p(\cdot)$ is the generalized Wiener-It\^o stochastic integral. Then the map $I_p$ provides a linear isometry between $\mathfrak{H}^{\odot p}$ (equipped with the norm $\frac{1}{\sqrt{p!}}\|\cdot\|_{\mathfrak{H}^{\otimes p}}$) and $\mathcal{H}_p$, namely, for any $f,g\in\mathfrak{H}^{\bigodot p}$ and $q\geq1$
$$\mathbb{E}\left[I_{q}(f) I_{q}(g)\right]=q!\langle f, g\rangle_{\mathfrak{H}^{\otimes q}},$$
to regulate that $\mathcal H_{o}=R$ and $I_{0}(x)=x$.\par
The $ Ornstein-Uhlenbeck $ operator L is defined by $LF=-\delta DF$. If $F=I_{q}(f_{q})$ is in the q-th Wiener chaos of $G$, $f_{q} \in \mathfrak{H}^{\odot q}$, then $LF=-qF$.\par
We will adopt two fundamental Theorems: Theorem 6.2.3 of \cite{A18} and Theorem 3.1 of \cite{E15}.\par
\section{Proof of Theorem \ref{D1}}\label{I2}
This section will apply the following well-known Lemma.
\begin{lem}\label{003}
\begin{description}
\item[(1)]Let the continuous function $f(x)\geq0$, if $\int^{\infty}_{0}f(x)\mathrm{d} x=0$, then ~$f(x)\equiv0$.
\item[(2)] Let the continuous binary function $f(x,y)\geq0$, if $\int^{\infty}_{0}\int^{\infty}_{0}f(x,y)\mathrm{d} x\mathrm{d }y=0$, then $f(x,y)\equiv0$.
\item[(3)] \cite[P56]{A15} Let $f$ be a measurable function of the measure space $(\Omega,\mathscr{F},\mu)$, if
$\int_{\Omega}f \mathrm{d} \mu=0$ and $f\geq0$ a.e, then $f=0$ a.e.
\end{description}
\end{lem}

\par
The following proves that the 9 types of zero-mean Gaussian processes in Theorem \ref{D1} satisfy hypothesis \ref{as01} and hypothesis   \ref{as02}.
\begin{enumerate}
\item \begin{proof}
When $t=s$, it can be seen clearly that hypothesis \ref{as01} is true. As shown below, hypothesis \ref{as02} is proved true.
According to the property 2.1 of \cite{d9}, we have
$$\operatorname{Cov}\left(X_{t^{2H}}^{K}, X_{s^{2H}}^{K}\right)=\int_{0}^{\infty}\left(1-e^{-\theta t^{2H}}\right)\left(1-e^{-\theta s^{2H}}\right) \theta^{-1-K} \mathrm{~\mathrm{d}} \theta.$$
Therefore, for any n-dimensional vector $\mathbf{a}=(a_{1},a_{2},\cdots, a_{n})^{\prime}$,
$$\mathbf{a}^{\prime}\mathbf{V}\mathbf{a}=\sum^n_{l=1}\sum^n_{j=1}a_{j}R(t_{j},t_{l})a_{l}=\int_{0}^{\infty}\lvert\sum^{n}_{j=1}\left(1-e^{-\theta t_{j}^{2H}}\right)a_{j}\rvert^{2} \theta^{-1-K} \mathrm{\mathrm{d}} \theta\geq 0.$$
When $\int_{0}^{\infty}\lvert\sum^{n}_{j=1}\left(1-e^{-\theta t_{j}^{2H}}\right)a_{j}\rvert^{2} \theta^{-1-K} \mathrm{\mathrm{d}} ~\theta=0$,  $\theta$ is taken as $1,2,\cdots,n-1$ in turn, and based on Lemma \ref{003}, we have\\
\begin{equation}\label{x12}
\begin{aligned}
&\sum^n_{j=1}(1-e^{-k t_{j}^{2H}})a_j=0(k=1,2,\cdots,n-1)\;\\
 \;&\sum^{n}_{j=1}\lim \limits_{\theta\rightarrow\infty}(1-e^{-\theta t_{j}^{2H}})a_j=0,
 \end{aligned}
\end{equation}
since
\begin{align*}
&\abs{\begin{array}{cccc}
1 & 1 & \cdots & 1 \\
1-e^{-t_{1}^{2 H}} & 1-e^{-t_{2}^{2H}} & \cdots & 1-e^{-t_{n}^{2 H}} \\
1-e^{-2 t_{1}^{2 H}} & 1-e^{-2 t_{2}^{2 H}} & \cdots & 1-e^{-2 t_{n}^{2 H}} \\
\cdots & \cdots & \cdots & \cdots \\
1-e^{-(n-1) t_{1}^{2 H}} & 1-e^{-(n-1) t_{2}^{2 H}} & \cdots & 1-e^{-(n-1) t_{n}^{2 H}}
\end{array}}
\\&=(-1)^{n-1}\prod_{1 \leq i<j \leq n}\left(e^{-t_{j}} - e^{-t_{i}}\right) \neq 0 .
\end{align*}
The only solution is zero for the homogeneous linear equation \eqref{x12}.
\end{proof}
\item \begin{proof}
When $s=t$, it is evident that hypothesis \ref{as01} is true.
As shown below, hypothesis \ref{as01}  is true. For any $x\in R$, there is
\begin{equation}\label{e1}
 \abs{x}^{2 H}=\frac{1}{c_{H}} \int_{0}^{\infty} \frac{1-e^{-u^{2} x^{2}}}{u^{1+2 H}} \mathrm{d} u,
\end{equation}
where $c_{H}=\int_{0}^{\infty}\left(1-e^{-u^{2}}\right) u^{-1-2 H}\mathrm{d} u<\infty$.\par
For any n-dimensional vector $\mathbf{a}=(a_{1},a_{2},\cdots,a_{n})^{\prime}$, and according to the property 1.6 on page 9 of \cite{A12}, it can be known that:
\begin{align*}
\mathbf{a}^{\prime}\mathbf{V}\mathbf{a}&=\sum^n_{l=1}\sum^n_{j=1}a_{k}R(t_{k},t_{l})a_{l}\\&=\sum_{k, l=1}^{n} \frac{1}{2}\left(t_{k}^{2 H}+t_{l}^{2 H}-\lvert t_{k}-t_{l}\rvert^{2 H}\right) a_{k} a_{l}
\\&=\frac{1}{2 c_{H}} \int_{0}^{\infty} \frac{\left(\sum_{k=1}^{n}\left(1-e^{-u^{2} t_{k}^{2}}\right) a_{k}\right)^{2}}{u^{1+2 H}} \mathrm{d} u
\\& \quad +\frac{1}{2 c_{H}} \sum_{d=1}^{\infty}\frac{2^{d}}{d!} \int_{0}^{\infty} \frac{\left(\sum_{k=1}^{n} t_{k}^{d} e^{-u^{2} t_{k}^{2}} a_{k}\right)^{2}}{u^{1-2 d+2 H}}\mathrm{d} u\\&\geq
\frac{1}{2 c_{H}} \int_{0}^{\infty} \frac{\sum_{d=1}^{n} \frac{2^{d}}{d!}\left(\sum_{k=1}^{n} t_{k}^{d} e^{-u^{2} t_{k}^{2}} a_{k}\right)^{2} }{u^{2 H+1-2 d}}\mathrm{d}u \geq 0.
\end{align*}

Based on Lemma \ref {003}, if the last equation is equal to 0, then there is
\begin{equation}
\sum_{d=1}^{n}\left(\sum_{k=1}^{n}t_{k}^{d} e^{-u^{2} t_{k}^{2}} a_{k}\right)^{2}=0,
\end{equation}
namely,
\begin{equation}\label{x15}
t_{1}^{k} e^{-u^2t_{1}^{2}} a_{1}+t_{2}^{k} e^{-u^2t_{2}^{2}}a_{2}+\cdots+t_{n}^{k} e^{-u^2 t_{n}^{2}} a_{n}=0\qquad(k=1,2\cdots,n),
\end{equation}
as$$
\abs{\begin{array}{cccc}
t_{1} e^{-u^{2} t_{1}^{2}} & t_{2} e^{-u^{2} t_{2}^{2}} & \cdots & t_{n} e^{-u^{2} t_{n}^{2}} \\
t_{1}^{2} e^{-u^{2} t^2_{1}} & t_{2}^{2} e^{-u_{2}^{2} t_{2}^{2}} & \cdots & t_{n}^{2} e^{-u^{2} t_{n}^{2}} \\
\vdots & \vdots & & \vdots \\
t_{1}^{n} e^{-u^{2} t} & t_{2}^{n} e^{-u_{2}^{2} t_{2}^{2}} & \cdots & t_{n}^{n} e^{-u^{2} t_{n}^{2}}
\end{array}}= \prod_{k=1}^n t_{k} e^{-u^{2} t_{k}^{2}} \prod_{1 \leq i<j \leq n}\left(t_{j}-t_{i}\right) \neq 0.
$$\par
Therefore, the only solution is zero for the homogeneous linear equation \eqref{x15}.
\end{proof}
\item \begin{proof}
When $t=s$, it can be seen clearly that hypothesis \ref{as01} is true.
As shown below, that hypothesis \ref{as02} is true.\par
For any n-dimensional vector $\mathbf{a}=(a_{1},a_{2},\cdots,a_{n})^{\prime}$, based on equation (\ref{e1}), there is
\begin{align*}
\mathbf{a}^{\prime}\mathbf{V}\mathbf{a}&=\sum_{i=1}^{n} \sum_{j=1}^{n} a_{i}\left(\lvert t_{i}+t_{j}\rvert^{2 H}-\lvert t_{i}-t_{j}\rvert^{2 H}\right) a_{j}\\&=
\sum_{i=1}^{n} \sum_{j=1}^{n} a_{i}a_{j} \frac{1}{c_H} \int_{0}^{\infty} \frac{1-e^{-u^{2}\left(t_{i}+t_{j}\right)^{2}}-1+e^{-u^{2}\left(t_{i}-t_{j}\right)^{2}}}{u^{2 H+1}} \mathrm{d}u\\&=
\frac{1}{c_{H}} \int_{0}^{\infty} \frac{\sum_{d =1}^{\infty} \frac{2^{2 d-1}}{(2 d-1)!} \cdot \sum^{n}_{i=1} \sum_{j=1}^{n} a_{i} t_{i}^{2 d-1} e ^{-u^2t_i^2}a_{j} t_{j}^{2 d-1}e ^{-u^2t_j^2}}{u^{2 H-4 d+3}}  \mathrm{d}u
\\&\geqslant \frac{1}{c_{H}} \int_{0}^{\infty} \sum_{d=1}^{n} \frac{2^{2d-1}}{(2d-1) !}\left(\sum_{i=1}^{n} a_{i} t_{i}^{2 d-1}e^{-u^2t_i^2}\right)^{2} \mathrm{d}u\geq 0 .
\end{align*}
Lemma \ref {003} means that the last equation is equal to 0, then there is
\begin{equation}
\sum_{d=1}^{n}\left(\sum_{k=1}^{n}t_{k}^{2d-1} e^{-u^{2} t_{k}^{2}} a_{k}\right)^{2}=0,
\end{equation}
namely,
\begin{equation}\label{3.6}
t_{1}^{2k-1} e^{-u^2t_{1}^{2}} a_{1}+t_{2}^{2k-1}e^{-u^2t_{2}^{2}}a_{2}+\cdots+t_{n}^{2k-1} e^{-u^2 t_{n}^{2}} a_{n}=0\,(k=1,2,\cdots,n),
\end{equation}
as
\begin{align*}
&\abs{\begin{array}{cccc}
t_{1} e^{-u^{2} t_{1}^{2}} & t_{2} e^{-u^{2} t_{2}^{2}} & \cdots & t_{n} e^{-u^{2} t_{n}^{2}} \\
t_{1}^{3} e^{-u^{2} t^2_{1}} & t_{2}^{3} e^{-u_{2}^{2} t_{2}^{3}} &\cdots & t_{n}^{3} e^{-u^{2} t_{n}^{2}} \\
\vdots & \vdots & & \vdots \\
t_{1}^{2n-1} e^{-u^{2} t} & t_{2}^{2n-1} e^{-u_{2}^{2} t_{2}^{2}} &\cdots & t_{n}^{2n-1} e^{-u^{2} t_{n}^{2}}
\end{array}}\\&=\prod_{i=1}^{n} t_{i} \prod_{i=1}^{n} e^{-u^{2} t_{i}^{2}} \prod_{1\leq i<j \leq n}\left(t_{j}^{2}-t_{i}^{2}\right)  \neq 0.
\end{align*}
Therefore, the homogeneous linear equation \eqref{3.6} has only solution of zero.
\end{proof}

\item \begin{proof}
When $s=t$ and $c\geq1$, then $R(t,s)=t^{2H}\leq ct^{2H}$, hypothesis \ref{as01} is true. As shown below, the strict positive definition of covariance matrix is proved true.\par
 For any n-dimensional vector $\mathbf{a}=(a_{1},a_{2},\cdots,a_{n})^{\prime}$, according to equation ~(\ref{e1}), there is
\begin{align*}
\mathbf{a}^{\prime}\mathbf{V}\mathbf{a}&=\frac{1}{2} \sum_{k=1}^{n} \sum_{l=1}^{n} a_{k} a_{l}\left(2 t_{k}^{2 H}+2 t_{l}^{2 H}-\left(t_{k}+t_{l}\right)^{2 H}-\lvert t_{k}-t_{l}\rvert^{2 H}\right)
%\\&=\frac{1}{2 c_{H}} \sum_{k=1}^{n} \sum_{l=1}^{n} a_{k} a_{l} \int_{0}^{\infty} \frac{2\left(1-e^{-u^{2} t_{k}^{2}}-e^{-u^{2} t_{l}^{2}}\right)+e^{-u^{2} t_{k}^{2}-u^{2} t_{l}^{2}}(e^{-2 u^{2} t_{k} t_{l}+2 u^{2} t_{k} t_{l}})}{u^{2 H+1}} \mathrm{d}u
\\&=
\frac{1}{c_{H}} \int_{0}^{\infty} \frac{\left(\sum_{k=1}^{n}\left(1-e^{-u^{2} t_{k}^{2}}\right) a_{k}\right)^{2}}{u^{2 H+1}}  \mathrm{d}u \\&\quad+\frac{1}{c_{H}} \int_{0}^{\infty} \frac{\sum_{d=1}^{\infty} \frac{2^{2 d}}{(2 d) !}\left(\sum_{k=1}^{n} e^{-u^{2} t_{k}^{2}} t_{k}^{2 d}a_{k}\right)^{2}}{u^{2 H+1-4 d}}  \mathrm{d}u
\\&
\geqslant \frac{1}{c_{H}} \int_{0}^{\infty}\frac{\sum_{d=1}^{n} \frac{2^{2d}}{(2d)!}\left(\sum_{k=1}^{n} e^{-u^{2} t_{k}^{2}} t_{k}^{2d}a_{k}\right)^{2} }{u^{2H+1-4d}}\mathrm{ d}u\geq 0,
\end{align*}
Based on Lemma~\ref {003}, if the last equation is equal to 0, then there is
\begin{equation}
\sum_{d=1}^{n}\left(\sum_{k=1}^{n}t_{k}^{2d} e^{-u^{2} t_{k}^{2}} a_{k}\right)^{2}=0 .
\end{equation}
Therefore,
\begin{equation}\label{3.8}
\sum_{k=1}^{n}t_{k}^{2d} e^{-u^{2} t_{k}^{2}} a_{k}=0 \quad(d=1,2,\cdots,n),
\end{equation}
consider that
$$\abs{\begin{array}{cccc}
t_{1}^2 e^{-u^{2} t_{1}^{2}} & t_{2}^2 e^{-u^{2} t_{2}^{2}} & \cdots & t_{n}^2 e^{-u^{2} t_{n}^{2}} \\
t_{1}^{4} e^{-u^{2} t^2_{1}} & t_{2}^{4} e^{-u^{2} t_{2}^{2}} & \cdots & t_{n}^{4} e^{-u^{2} t_{n}^{2}} \\
\vdots & \vdots & & \vdots \\
t_{1}^{2n} e^{-u^{2} t_{1}^2} & t_{2}^{2n} e^{-u^{2} t_{2}^{2}} & \cdots & t_{n}^{2n} e^{-u^{2} t_{n}^{2}}
\end{array}}=
 \prod_{k=1}^{n} t_{k}^{2} e^{-u^{2} t_{k}^{2}} \prod_{1 \leq i<j \leq n}\left(t^{2}_{j}-t^{2}_{i}\right) \neq 0.$$
Therefore, the solution is zero for the homogeneous linear equation \eqref{3.8}.
\end{proof}
\item \begin{proof}
When $t=s$, it can be seen clearly that hypothesis \ref{as01} is true. The following will prove that hypothesis \ref{as02} is true.\par
For any n-dimensional vector $\mathbf{a}=(a_{1},a_{2},\cdots,a_{n})^{\prime}$, base on Theorem 1.1 of \cite{A13}, there is
\begin{align*}
\mathbf{a}^{\prime}\mathbf{V}\mathbf{a}&=\sum_{j, k=1}^{n} a_{j} R\left(t_{j}, t_{k}\right) a_{l}\\&=2H \int_{0}^{\infty} \int_{0}^{\infty}\lvert\sum_{j=1}^{n} a_{j} \mathbf{1}_{\left\{u \leq t_{j}\right\}}\mathbf{1}_{\left\{r\leq\left(t_{j}+u\right)^{2H-1}\right\}}\rvert^{2} \mathrm{d }r \mathrm{d} u \geq 0.
\end{align*}
When
\begin{equation}
\sum_{j, k=1}^{n} a_{j} R\left(t_{j}, t_{k}\right) a_{l}=2H \int_{0}^{\infty} \int_{0}^{\infty}\lvert\sum_{j=1}^{n} a_{j} \mathbf{1}_{\left\{u \leq t_{j}\right\}} \mathbf{1}_{\left\{r\leq\left(t_{j}+u\right)^{2H-1}\right\}}\rvert^{2} \mathrm{ d}r \mathrm{d}u = 0,
\end{equation}

take the lower limit $0$, $(t_{n-1}+u)^{2H-1}$,\quad$(t_{n-2}+u)^{2H-1}, \cdots, (t_{2}+u)^{2H-1}$ of the integral variable r, and then from Lemma ~\ref{003}-(3), we have
\begin{equation}\label{3.2}
\left\{\begin{array}{l}
\sum_{j=1}^{n} a_{j}=0 \\
\sum_{j=1}^{n-1} a_{j}=0 \\
\vdots \\
\sum_{j=1}^{2} a_{j}=0  \\
a_{1}=0.
\end{array}\right.
\end{equation}
Therefore, the only solution is zero for the homogeneous linear equation \eqref{3.2}.
\end{proof}
\item\begin{proof}
When $t=s$, it can be seen clearly  that hypothesis \ref{as01} is true. As shown below, The hypothesis \ref{as02} is true.\\
For any n-dimensional vector  $\mathbf{a}=(a_{1},a_{2},\cdots,a_{n})^{\prime}$, according to Theorem 1.1 of \cite{A13}, there is
\begin{align*}
\mathbf{a}^{\prime}\mathbf{V}\mathbf{a}&=\sum_{j, k=1}^{n} a_{j} R (t_{j}, t_{l}) a_{l}\\&=c_{2H} \int_{0}^{\infty} \int_{0}^{\infty} y^{-2H} e^{-y r} \mathbb{E}\lvert\sum_{j=1}^{n} a_{j} \mathbf{1}_{\left\{r \leq t_{j}\right\}} e^{i y \eta_{t_{j}}}\rvert^{2}  \mathrm{d}r \mathrm{d}y \geq 0,
\end{align*}
where $c_{2H}=\left(\int_{0}^{\infty} \frac{1-e^{-y}}{y^{2H+1}}  \mathrm{y}\right)^{-1}=\frac{2H}{\Gamma(1-2H)}>0.$
when
\begin{equation}
c_{2H} \int_{0}^{\infty} \int_{0}^{\infty} y^{-2H} e^{-y r} E\lvert\sum_{j=1}^{n} a_{j} \mathbf{1}_{\left\{r \leq t_{j}\right\}} e^{i y \eta_{t_{j}}}\rvert^{2}  \mathrm{d}r \mathrm{d}y =0,
\end{equation}
the similar practice to equation \eqref{3.2} is adopted to take the lower limit $0$, $t_1, t_2, \cdots, t_{n-1}$ of the integral variable r, then based on Lemma \ref{003}-(3), there is
\begin{equation}\label{05}
\left\{\begin{array}{l}
\sum_{j=1}^{n} e^{i y \eta_{t_{j}}}=0 \\
\sum_{j=2}^{n} e^{i y \eta_{t_{j}}}=0  \\
\vdots \\
\sum_{j=n-1}^{n-2} e^{i y \eta_{t_{j}}}=0  \\
a_{n}e^{i y \eta_{t_{n}}}=0.
\end{array}\right.
\end{equation}
Therefore, the only solution is zero for homogeneous linear equation \eqref{05}.
\end{proof}
\item\begin{proof}
When $t=s$, it can be seen clearly that hypothesis \ref{as01} is true.\par
$R(s,t)$ can be decomposed into
\begin{align*}
R(s,t)=&(1-K)\frac{1}{2}[t^{2H}+s^{2H}-\abs{t-s}^{2H}]\\&+K\frac{2H}{2\Gamma(1-2H)}\cdot \frac{\Gamma(1-2H)}{2H}[t^{2 H}+s^{2 H}-(t+s)^{2H}].
\end{align*}
Judging from the conclusions of 1 and 2 in Theorem ~\ref{D1}, it can be known that the covariance matrix corresponding to 7 is obtainable through the introduction of two positive definite matrices.
\end{proof}
\par
\item
\begin{proof}
When $t=s$, it can be seen clearly that hypothesis \ref{as01} is true in (a) and (b). The following will prove that hypothesis \ref{as02} is true.\par
(a). When $H\in(0,1)$,$K\in(0,1)$ and $H K<1$, based on property 2.1 of \cite{A16} and property 1.5 of \cite{A18}, we have\par
\begin{align*}
&\sum_{i=1}^{n} \sum_{j=1}^{n} a_{i} a_{j} R\left(t_{i}, t_{j}\right)\\=&\frac{K}{\Gamma(1-K)} \int_{0}^{\infty} \sum_{i=1}^{n} \sum_{j=1}^{n} a_{i} e^{-x t_{i}^{2 H}} a_{j} e^{-x t_{j}^{2 H}}\left(e^{x\left(t_{i}^{2 H}+t_{j}^{2 H}-\lvert t_{j}-t_{i}\rvert^{2 H}\right)}-1\right) x^{-1-k} \mathrm{d}x
\\=&\int_{0}^{\infty} \sum_{i=1}^{n} \sum_{j=1}^{n} a_{i} e^{-x t_{i}^{2 H}}\left(t_{i}^{2 H}+t_{j}^{2 H}-\lvert t_{j}-t_{i}\rvert^{2 H}\right) a_{j} e^{-x t_{j}^{2 H}} x^{-K} \mathrm{d}x
\\ & +\int_{0}^{\infty} \sum_{i=1}^{n} \sum_{j=1}^{n} a_{i} e^{-x t_{i}^{2 H}} \frac{\left(t_{i}^{2 H}+t_{j}^{2 H}-\lvert t_{j}-t_{i}\rvert^{2 H}\right)^{2}}{2 !} a_{j} e^{-x t_{j}^{2 H}} x^{1-k} \mathrm{d}x +\cdots
\\ \geqslant & \int_{0}^{\infty} \sum_{i=1}^{n} \sum_{j=1}^{n} a_{i} e^{-x t_{i}^{2 H}}\left(t_{i}^{2 H}+t_{j}^{2 H}-\lvert t_{j}-t_{i}\rvert^{2H}\right) a_{j} e^{-x t_{j}^{2 H}} x^{-k}  \mathrm{d}x \geq0.
\end{align*}
According to Lemma \ref{003}, if the last equation is 0, then there is
\begin{equation}\label{008}
\sum_{i=1}^{n} \sum_{j=1}^{n} a_{i} e^{-x t_{i}^{2 H}}\left(t_{i}^{2 H}+t_{j}^{2 H}-\lvert t_{j}-t_{i}\rvert^{2H}\right) a_{j} e^{-x t_{j}^{2 H}}=0.
\end{equation}
According to the conclusion of 2 of Theorem \ref{D1}, if equation (\ref{008}) is equals to 0, then $a_{i} e^{-x t_{i}^{2 H}}=0(i=1,\cdots,n)$, Besides, since $e^{-x t_{i}^{2 H}}\neq 0$, so $a_{i}=0(i=1,\cdots,n)$. \par

(b). When $H\in(0,1)$,$K\in(1,2)$ and $HK<1$, $R(t,s)$ can be decomposed into:
$$
R(t,s)=\frac{1}{2^K}\left[\left(\left(s^{2 H}+t^{2H}\right)^{K}-t^{2 H K}-s^{2 HK}\right)+\left(t^{2H K}+s^{2H K}-\lvert t-s\rvert^{2 H K}\right)\right],
$$

Based on conclusions of 1 and 2 of Theorem \ref{D1}, it can be known that the covariance matrix corresponding to 8-(b) can be obtained by adding two strictly positive definite matrices.
\end{proof}
\item \begin{proof}
When t=s, it can be seen clearly  that hypothesis ~\ref{as01} is true in (a) and (b). As shown below, hypothesis \ref{as02} is proved true.\par
(a). $H \in (0, 1), \, K \in(1,2),H K<1$, $R(t,\,s)$ can be decomposed into
\begin{align*}
R(t, s)&=\left[\left(s^{2H}+t^{2 H K}\right)^{K}-t^{2 H K}-s^{2 H K}\right]\\&\quad+\left[t^{2 H K}+s^{H K}-\frac{1}{2}\left[(t+s)^{2 H K}+\vert t-s\vert^{2 H K}\right]\right].
\end{align*}
From the conclusions of 1 and 4 of Theorem \ref{D1}, we can know that the covariance matrix corresponding to 9-(a) can be obtained by adding two strictly positive definite matrices.\par
(b).$H\in(0,\frac{1}{2})$, $K\in(0,\frac{1}{2})$, $R(t,s)$ can be decomposed into
$$ R(s,t)=\frac{\left(s^{2 H}+t^{2 H}\right)^{K}-(t+s)^{2HK}}{2}+\frac{\left.\left(s^{2H}+t^{2 H}\right)^{K}-\vert t-s \vert^{2 H K}\right)}{2}.$$
According to 3.12 in \cite{A22} and the conclusion of 2 of Theorem \ref{D1}, the covariance matrix corresponding to 9-(b) can be obtained by introducing two strictly positive definite matrices.
\end{proof}
\end{enumerate}
\section{Proof of Theorem \ref{D4}}
Before proving Theorem, the following Lemmas are obtained according to the previous introduction and preliminary knowledge.
\begin{lem}\label{zengliang est}
 When hypothesis \ref{as02} is true, for any $\mathbf{t}^{\prime}\in\mathbb{R}^n$, if $\mathbf{G}_{\mathbf{t}}\sim N(0,\mathbf{V})$, the following conclusions can be drawn obtained in combination with \cite{A04}:\par
$(A_1)$.~$\mathbf{t}^{\prime} \mathbf{V}^{-1} \mathbf{G}_{\mathbf{t}}\sim N(\mathbf{0},\mathbf{t}^{\prime} \mathbf{V}^{-1}\mathbf{t})$,\, $\mathbb{E}\left(\left(\mathbf{t}^{\prime} \mathbf{V}^{-1} \mathbf{G}_{t}\right)^{4}\right)=3\left(\mathbf{t}^{\prime} \mathbf{V}^{-1} \mathbf{t}\right)^{2}$.\par
$(A_2)$.~$(\mathbf{G_t})^{'}\mathbf{ V}^{-1}\mathbf{G_t}\sim\chi^2_n$,\, $\mathbb{E}((\mathbf{G_t})^{\prime} \mathbf{V}^{-1}\mathbf{G_t})=n, \, \mathbb{E}((\mathbf{\mathbf{G_t}})^{'} \mathbf{V}^{-1}\mathbf{G_t})^2=n(n+2)$.\par
$(B_1)$.~$\mathbb{E}\left(\mathbf{G_{t}}^{\prime} \mathbf{V}^{-1} \mathbf{G_{t}}\left(\mathbf{t}^{\prime} \mathbf{V}^{-1}\mathbf{G_{t}}\right)\right)=0$.
\par
$(B_2)$.~$\mathbb{E}\left(\mathbf{G_{t}}^{\prime} \mathbf{V}^{-1} \mathbf{G_{t}}\left(\mathbf{t}^{\prime} \mathbf{V}^{-1}\mathbf{G_{t}}\right)^{2}\right)=
(n+2)\left(\mathbf{t}^{\prime} \mathbf{V}^{-1} \mathbf{t}\right).$
\end{lem}
\begin{proof}
Given that $\mathbf{G}_t$ conforms to a multivariate normal distribution, conclusion $(A_1)$ is supported. The conclusion $(A_2)$ is reached according to Theorem \ref{D4} in Chapter 3 of \cite{A04}. As shown below, conclusions $(B_1)$ and $(B_2)$ are validated:\par
Let $\mathbf{V}^{-1/2}\mathbf{G_t}=[y_1,y_2,\cdots,y_n]^{\prime}$, then $\mathbf{V}^{-1/2}\mathbf{G_t}\sim N_n(\mathbf{0},I_n)$. Therefore, $y_1,y_2,\cdots,y_n$ are independent identically distributed and $y_{i}\sim N(0,1)$. Let $\mathbf{t^{\prime}}\mathbf{V}^{-1/2}=[a_1,a_2,\cdots,a_n]$, then for the conclusion $(B_1)$, there is
\begin{align*}
\mathbb{E}\left[\mathbf{G_{t}}^{\prime} \mathbf{V}^{-1} \mathbf{G}\left(\mathbf{t}^{\prime} \mathbf{V}^{-1} \mathbf{G_{t}}\right)\right]&=\mathbb{E}\left[\left(\sum_{i=1}^{n} y_{i}^{2}\right)\left(\sum_{i=1}^{n} a_{i} y_{i}\right)\right]\\
\\&=\sum_{i=1}^{n} a_{i} \mathbb{E} y_{i}^{3}+2\sum_{1\leq i<j \leq n} a_{i}\mathbb{E}\left(y_{i} y_{j}^{2}\right)
=0.
\end{align*}
Then, for the conclusion ($B_2$), there is
\begin{align*}
&\mathbb{E}\left(\mathbf{G_{t}}^{\prime} \mathbf{V}^{-1} \mathbf{G_{t}}\left(\mathbf{t}^{\prime} \mathbf{V}^{-1}\mathbf{G_{t}}\right)^{2}\right) \\ =&\mathbb{E} y_1^2(\sum_{i=1}^{n}a_{i}^{2}y_{i}^2)+\mathbb{E} y_2^2(\sum_{i=1}^{n}a_{i}^{2}y_{i}^2)+\cdots+\mathbb{E} y_n^2(\sum_{i=1}^{n} a_i^{2}y_i^2)\\& +2\sum^{n}_{j=1}(\sum_{1\leq k<i\leq n} a_i a_j\mathbb{E} (y_j^2 y_k y_i))
\\=&(n+2)\sum_{i=1}^n a_i^2
=(n+2)\left(\mathbf{t}^{\prime} \mathbf{V}^{-1} \mathbf{t}\right).
\end{align*}
\end{proof}
\begin{remark}
For the proofs of conclusions ($B_1$) and ($B_2$), the standard technique of completing the square of Theorem 3.2 of \cite{E9} is also applicable. The proof method used in the thesis is advantageous in conciseness.
\end{remark}
\begin{lem}
Hypothesis \ref{as01} and \ref{as02} are true, then the estimator $\hat{\mu}_n$ of $\mu$ is an unbiased estimate and convergence in $L^2$.
\end{lem}
\begin{proof}
By (2) and (4), we have
\begin{equation}\label{x25}
\hat{\mu}_n=\mu+\sigma \frac{\mathbf{t}^{\prime}\mathbf{V}^{-1} \mathbf{G}_{\mathbf{t}}}{\mathbf{t}^{\prime}\mathbf{V}^{-1} \mathbf{t}}.
\end{equation}
According to ($A_1$), it can be known that $\frac{\mathbf{t}^{\prime}\mathbf{V}^{-1} \mathbf{G}_{\mathbf{t}}}{\mathbf{t}^{\prime} V^{-1} \mathbf{t}}\sim N(0,\frac{1}{\mathbf{t}^{\prime} \mathbf{V}^{-1} \mathbf{t}})$. Therefore, $\mathbb{E}\hat{\mu}_n=\mu$, which means $\hat{\mu}_n$  is an unbiased estimate of $\mu$ .Also,\\
\begin{align*}
\mathbb{E}\left[(\hat{\mu}_n-\mu)^{2}\right]
=\frac{\sigma^{2}}{\mathbf{t}^{\prime}\mathbf{V}^{-1} \mathbf{t}}.
\end{align*}

If the hypothesis \ref{as02} is true, it can be known that $\mathbf{V}$ is a strictly positive definite matrix. In this circumstance, all eigenvalues $\lambda>0$ of $\mathbf{V}$.
Let $\lambda_n$ be the maximum eigenvalue of $\mathbf{V}$, then it can be known that $\frac{1}{\lambda_n}$ is the minimum eigenvalue of $\mathbf{V}^{-1}$ according to
$$\mathbf{V}\alpha=\lambda\alpha \Longrightarrow \mathbf{V}^{-1}\alpha=\frac{\alpha}{\lambda}.$$

If the hypothesis \ref {as01} is true, according to 1.6.3 in \cite {A04}, there is\\
$$\mathbf{t}^{\prime}\mathbf{V}^{-1}\mathbf{t} \geqslant \frac{\mathbf{t}^{\prime}\mathbf {t}}{\lambda_{n}}=\frac{\sum^{n}_{i=1}t_{i}^{2}}{\lambda_{n}} \geqslant \frac{\sum^{n}_{i=1}t_{i}^{2}}{\sum_{i=1}^{n}t_{i}^{\beta}}.$$
Since $\underset{{n\rightarrow\infty}}{\lim}\frac{t_n^{\beta}}{t_n^2}=0$, therefore, according to $stolz$ equation, we have
$$\underset{{n\rightarrow\infty}}{\lim}\frac{\sum^{n}_{i=1}t_{i}^{\beta}}{\sum_{i=1}^{n}t_i^{2}}=0.$$
so \\
\begin{equation}\label{w2}
\mathbb{E}(\hat{\mu}_n-\mu)^{2}=\frac{\sigma^{2}}{\mathbf{t}^{\prime} \mathbf{V}^{-1} \mathbf{t}} \leqslant\sigma^2\frac{\sum^{n}_{i=1}t_{i}^{\beta}}{\sum_{i=1}^{n}t_i^{2}} \rightarrow 0\quad as\quad n \rightarrow \infty.
\end{equation}
\end{proof}
\begin{lem}\label{e32}
If the hypothesis \ref{as02} is true,  $\hat{\sigma}_n^2$ is an asymptotic unbiased estimator and $L^2$-consistency.
\end{lem}
\begin{proof}
By (2) and (5), we have
\begin{equation}
\hat{\sigma}^{2}_n=\frac{\sigma^{2}}{n}\left( \mathbf{G_{t}}^ { \prime } \mathbf{V} ^ {-1} \mathbf{G_t}-\left(\frac{\left(\mathbf{t}^{\prime} \mathbf{V}^{-1} \mathbf{G_{t}}\right)^{2}}{\mathbf{t}^{\prime}  \mathbf{V}^{-1} \mathbf{t}}\right)\right).\\
\end{equation}
According to conclusion $(A_2)$ and conclusion $(B_2)$, there is
\begin{equation}
\mathbb{E}\hat{\sigma}^{2}_n
=\frac{(n-1) \sigma^{2}}{n}\rightarrow\sigma^2\quad n\rightarrow \infty.
\end{equation}
According to conclusion ($A_2$) and conclusion ($B_2$) , there is
\begin{align*}
&\mathbb{E}[\left(\hat{\sigma}^{2}_n\right)^{2}] \\=&\frac{\sigma^{4}}{n^{2}}\left(\mathbb{E}\left((\mathbf{G}_{t}^{\prime} \mathbf{V}^{-1}\left(\mathbf{G}_{t}\right))^{2}-\frac{2 \mathbb{E}\left(\mathbf{G}_{t}^{\prime} \mathbf{V}^{-1} \mathbf{G}_{t}\left(\mathbf{t}^{\prime}\mathbf{V}^{-1} \mathbf{G}_{t}\right)^{2}\right)}{\left(\mathbf{t}^{\prime} \mathbf{V}^{-1} \mathbf{t}\right)}+\frac{\mathbf{t}^{\prime}\mathbf{V}^{-1} \mathbb{E}\left(\mathbf{G}_{t} \mathbf{G}_{t}^{\prime}\right) \mathbf{V}^{-1} \mathbf{t}}{\left(\mathbf{t}^{\prime}\mathbf{V}^{-1} \mathbf{t}\right)^{2}}\right) \right)\\
=&\frac{\sigma^{4}}{n}(n(n+2)-2(n+2)+3)
=\frac{\sigma^{4}}{n}\left(n^{2}-1\right),
\end{align*}
namely,
\begin{equation}
\mathbb{E}[\left(\hat{\sigma}^{2}_n\right)^{2}]=\frac{\sigma^{4}}{n}\left(n^{2}-1\right).
\end{equation}
Therefore
\begin{equation}\label{x30}
\mathbb{E}\left(\lvert\hat{\sigma}^{2}_n-\sigma^{2}\rvert^{2}\right)=\mathbb{E}(\hat{\sigma}^{2}_n)^2-2\sigma^2 \mathbb{E}\hat{\sigma}^{2}_n+\sigma^{4}=\frac{2 n-1}{n^{2}} \sigma^{4}\rightarrow 0 \quad as\quad n\rightarrow \infty.
\end{equation}
\end{proof}

\noindent {\bf 4. Proof of Theorem \ref{D4}.1}
To prove that $\hat{\mu}_n$ converges almost everywhere, through the Borel-Cantelli Lemma,  it is sufficient to prove
\begin{equation}\label{x31}
\sum_{n=1}^{\infty} \mathbf{P}\left(\lvert\hat{\mu}_n-\mu\rvert>\epsilon\right)<\infty.
\end{equation}
%\frac{1}{n^{\gamma}}c \sigma^{q} h^{(H-1) q} n^{q \gamma+(H-1) q}
According to equation (\ref{w2}), there is
\begin{equation}
\mathbb{E}(\hat{\mu}_n-\mu)^{2}\leq\sigma^2\frac{\sum_{i=1}^{n}t_{i}^{\beta}}{\sum_{i=1}^{n}t_{i}^{2}}\leq\frac{2\sigma^{2}t_{n}^{\beta}}{t^{2}_{n/2}}=\frac{2\sigma^{2}c_{2}^{\beta}n^{\alpha \beta}}{c_{1}^{2}2^{2\alpha}n^{2\alpha}}\leq c n^{\alpha(\beta-2)}.
\end{equation}
Pick $\epsilon>0$,  Chebyshev inequality is used.
\begin{align*}
\mathbf{P}\left(\lvert\hat{\mu}_n-\mu\rvert>\epsilon\right) &\leq \epsilon^{-q} \mathbf{E}\left(\lvert\hat{\mu}_n-\mu\rvert^{q}\right) \\
&\leq c\epsilon^{-q}\left(\mathbf{E}\left(\lvert\hat{\mu}_n-\mu\rvert^{2}\right)\right)^{q / 2}\\
&\leq c n^{\frac{q\alpha(\beta-2)}{2}}.
\end{align*}
When $q$ is large enough to the extent that $\frac{(\beta-2)\alpha q}{2}<-1$, the convergence of equation \eqref{x31}  is achievable. It can be verified that when the growth of $t_n$ is accelerated, the same conclusion can be drawn. For example, the growth of $t_n$ is in the range of $(c_1a^n,c_2a^n)$, where $a>1$.\par
Similarly, it can be proved that $\hat{\sigma}^2_n$ converges almost everywhere. Based on Lemma \ref{e32}, there is
\begin{align*}
P\left(\lvert\hat{\sigma}^{2}_n-\sigma^{2}\rvert>\epsilon\right) & \leq \epsilon^{-q} \mathbb{E}\left(\lvert\hat{\sigma}^{2}_n-\sigma^{2}\rvert^{q}\right) \\ & \leq c \epsilon^{-q}\left(\mathbb{E}\left(\rvert\hat{\sigma}^{2}_n-\sigma^{2}\rvert^{2}\right)\right)^{q / 2}\\& % \leq c n^{q \gamma}\left(\frac{2}{n}\right)^{q / 2} \sigma^{2 q} \\ & \leq c n^{q(\gamma-1 / 2)} . \end{align*}
\leq c n^{-\frac{q}{2}}.
\end{align*}
When q is large enough to the extent that $\frac{q}{2}<1$, according to the Borel-Cantelli Lemma, equation (7) is true.\\
\noindent {\bf 4. Proof of Theorem \ref{D4}.2}
From equation \eqref{x25} and the conclusion ($A_1$), it can be known that $Y_n \stackrel{l a w}{\longrightarrow}N(0,1)$. According to equation \eqref{x30}, there is
\begin{align*}
\mathbb{E}((Q_n)^2)=\frac{n}{2\sigma^2}\mathbb{E}(\hat{\sigma}_n^2-\sigma^2)^2=\frac{n}{2\sigma^2}\frac{2n-1}{n^2}\sigma^4\rightarrow 1\quad as\quad n\rightarrow \infty.
\end{align*}
since $Q_n\in\mathcal{H}_2$, which is similar to the proof method of Lemma 3.1 in sub-fractional Brownian motion in \cite {E10}, it can be known that $\|D Q_{n}\|_{\mathfrak{H}}^{2}=\frac{2\hat{\sigma}_n^2}{\sigma^2}$ . Then according to equation \eqref{x30}, there is
\begin{align*}
\mathbb{E}(\frac{2\hat{\sigma}^2_n}{\sigma^2}-2)^2=\frac{4}{\sigma^4}\mathbb{E}(\hat{\sigma}^2_n-\sigma^2)^2=\frac{4}{\sigma^4}\frac{2n-1}{n^2}\sigma^4\rightarrow 0\quad as\quad n\rightarrow\infty,
\end{align*}
According to the Theorem 2.1 of the \cite{E14} (fourth order moment Theorem), there is
$$Q_n=\frac{1}{\sigma^{2}} \sqrt{\frac{n}{2}}(\hat{\sigma}^{2}_n-\sigma^{2})\stackrel{l a w}{\longrightarrow} N(0,1).$$
Base on conclusions ($A_1$) and $(B_1)$, there is
\begin{align*}
 &\mathbb{E}(Q_{n} Y_{n})\\=&\frac{\sqrt{\mathbf{t}^{\prime} \mathbf{V}^{-1} \mathbf{t}}}{\sigma^{4}} \sqrt{\frac{n}{2}} \mathbb{E}\left(\hat{\mu}_{n}-\mu\right)\left(\hat{\sigma}_{n}^{2}-\sigma^{2}\right)\\
=&\frac{\sqrt{\mathbf{t}^{\prime}\mathbf{ V}^{-1} \mathbf{t}}}{\sigma} \sqrt{\frac{n}{2}}\mathbb{E} \left(\frac{\mathbf{t}^{\prime}\mathbf{ V}^{-1} \mathbf{G_t}}{\mathbf{t}^{\prime} \mathbf{V}^{-1}\mathbf{ t}}\left(\frac{1}{n}\left(\mathbf{G_{t}}^{\prime} \mathbf{V}^{-1} \mathbf{G_{t}}-\frac{\left(\mathbf{t}^{\prime} \mathbf{V}^{-1} \mathbf{G_{t}} t\right)^2}{\mathbf{t}^{\prime} \mathbf{V}^{-1} \mathbf{t}}\right)-1\right)\right)\\
=&\frac{\sqrt{\mathbf{t}^{-1} \mathbf{V}^{-1} \mathbf{t}}}{\sigma} \sqrt{\frac{n}{2}}\left[\frac{\mathbb{E}\left(\mathbf{t}^{\prime} \mathbf{V}^{-1}\mathbf{ \mathbf{G_{t}}} \mathbf{G_{t}}^{\prime} \mathbf{V}^{-1} \mathbf{G_{t}}\right)}{n \mathbf{t}^{\prime}\mathbf{ V}^{-1}\mathbf{ \mathbf{t}}}-\frac{\mathbb{E}\left(\mathbf{t}^{\prime} \mathbf{V}^{-1} \mathbf{G_{t}}\right)^{3}}{n\left(\mathbf{t}^{\prime} \mathbf{V}^{-1} \mathbf{t}\right)^{2}}-\frac{\mathbb{E}\left(\mathbf{t}^{\prime} \mathbf{V}^{-1} \mathbf{\mathbf{G_{t}}}\right)}{\mathbf{t}^{\prime} \mathbf{V}^{-1} \mathbf{t}}\right]\\
=&0 .
\end{align*}
Finally, Theorem \ref{D4}.2  is verified as true by Theorem 6.2.3 of \cite{A18}.\\
\noindent {\bf 4. Proof of Theorem \ref{D4}.3}
Based on Theorem 3.1 of \cite {E15}, it is sufficient to prove:\par
\begin{description}
\item[(i)]$\varphi(n):=\sqrt{\mathbb{E}\left[\left(1-\left\langle D \bar{Q}_{n},-D L^{-1} \bar{Q}_{n}\right\rangle_{\mathfrak{H}}\right)^{2}\right]}\rightarrow 0 \quad as \quad n\rightarrow \infty$.

\item[(ii)]$\left(\bar{Q}_{n}, \frac{1-\left\langle D\bar {Q}_{n},-D L^{-1} \bar{Q}_{n}\right\rangle_ \mathfrak{H}}{\varphi(n)}\right)\rightarrow (N_1,N_2)\quad as \quad n\rightarrow \infty,$
\end{description}

where ($N_1$, $N_1$) conforms to the bivariate standard normal distribution and the correlation coefficient is ~$\rho$.\par
Firstly, for part (i), according to equation \eqref{x30}, there is
$$\mathbb{E}\left[1-\left\langle D \bar{Q}_{n},-D L^{-1} \bar{Q}_{n}\right\rangle_{\mathfrak{H}}\right]^{2}=\mathbb{E}[1-\frac{1}{2}|| D Q_{n}||_{\mathfrak{H}}^{2}]=\frac{\mathbb{E}\left({\sigma^{2}}-\hat{\sigma}^{2}_n\right)^{2}}{\sigma^{4}}=\frac{2 n-1}{n^{2}},$$
Therefore,
$$\varphi(n)=\frac{\sqrt{2 n-1}}{n} \rightarrow 0 \quad as \quad n \rightarrow \infty.$$
For part (ii), since
$$var(\bar{Q}_{n})=var(Q_{n}) \rightarrow 1 \quad as \quad n\rightarrow\infty,$$
\begin{align*}
&var\left(\frac{1-\left\langle D \bar{Q}_{n},-D L^{-1}\bar{Q}_{n}\right\rangle}{\varphi(n)}\right)=\frac{var (\hat{\sigma}^{2}_n)}{\sigma^{2} \varphi^{2}(n)}=\frac{2 n-2}{2 n-1} \rightarrow 1 \quad as \quad n \rightarrow \infty.
\end{align*}
Therefore,
\begin{align*}
\rho&=\lim _{n\rightarrow\infty}cov(\bar{{Q}}_{n},\frac{1-\langle D\bar{Q}_{n},-D L^{-1} \bar{Q}_{n}\rangle}{\varphi(n)})\\&=\lim_{n\rightarrow\infty} \frac{\mathbb{E}(\bar{{Q}}_{n}(\sigma^{2}-\hat{\sigma}^{2}_n))}{ \varphi(n)\sigma^2}\\&= \lim_{n\rightarrow\infty}-\sqrt{\frac{n}{2}}[\mathbb{E}({\sigma^{2}}-\hat{\sigma}^{2}_n)^{2}-(\mathbb{E}({\sigma^{2}}-\hat{\sigma}^{2}_n))^2]\\&=
\lim_{n\rightarrow\infty}-\frac{\sqrt{n}(n-1)}{n\sqrt{n-\frac{1}{2}}}=-1.
\end{align*}

\bigskip

\bibliography{sn-bibliography}% common bib file

%-------------------------------------------------------------------------------------------------------
%-------------------------------------------------------------------------------------------------------
\end{document}